\documentclass[english,twocolumn]{paper}
\pdfoutput 1
\usepackage{times}
\usepackage{color}
\usepackage{mathptmx}
\usepackage{graphicx}
\graphicspath{{.}}
\usepackage{amsmath}
\usepackage{amsthm}
\usepackage{amssymb}
\usepackage[nothing]{algorithm}
\usepackage[noend]{algorithmic}
\usepackage{float}

\DeclareMathOperator*{\argmax}{argmax}

\newtheorem{theorem}{Theorem}

\newtheorem{lemma}{Lemma}

\newtheorem{corollary}{Corollary}

\usepackage{float}
\usepackage{hyperref}

\makeatletter
\def\blfootnote{\xdef\@thefnmark{}\@footnotetext}
\makeatother

\begin{document}

\title{Interdiction of a Markovian Evader}
\author{Alexander Gutfraind$^{1}$, Aric A. Hagberg$^{1}$,\\
David Izraelevitz$^{2}$, and Feng Pan$^{2}$}

\maketitle

\begin{abstract}
Shortest path network interdiction is a combinatorial optimization problem on an
activity network arising in a number of important security-related
applications.  It is classically formulated as a bilevel maximin
problem representing an ``interdictor'' and an ``evader''.  The evader
tries to move from a source node to the target node along a path of the least cost
while the interdictor attempts to frustrate this motion
by cutting edges or nodes.  The interdiction objective is to find the
optimal set of edges to cut given that there is a finite interdiction
budget and the interdictor must move first.  We reformulate the
interdiction problem for stochastic evaders by introducing a model in
which the evader follows a Markovian random walk guided by the
least-cost path to the target. This model can represent incomplete
knowledge about the evader, and the resulting model is a nonlinear $0-1$ optimization problem. We then introduce an optimization 
heuristic based on betweenness centrality that can rapidly find 
high-quality interdiction solutions by providing a global view of the network.
keyword: Network Interdiction; Stochastic Optimization; Guided Random Walk; Betweenness Centrality; LA-UR-08-06551
\end{abstract}

\section{Introduction}
\setcounter{footnote}{0} \blfootnote{$^{1}$ Theoretical Division, Los Alamos National Laboratory, Los Alamos, New Mexico USA 87545, \href{mailto:agutfraind.research@gmail.com}{agutfraind.research@gmail.com}, \href{mailto:hagberg@lanl.gov}{hagberg@lanl.gov}~~  $^{2}$ Risk Analysis and Decision Support Systems, Los Alamos National Laboratory, Los Alamos, New Mexico USA 87545, \href{mailto:izraelevitz@lanl.gov}{izraelevitz@lanl.gov}, \href{mailto:fpan@lanl.gov}{fpan@lanl.gov}.}%
Mathematical modeling of network interdiction originated
in the study of military supply chains and interdiction of transportation
networks~\cite{ghare-1971-optimal,mcmasters-1970-optimal}. The problem 
is currently studied in different classes of networks and in a variety of contexts, 
and finds applications in %defense against ballistic missiles, 
countering of nuclear proliferation programs~\cite{Morton-2007-models}, 
control of infectious diseases~\cite{Pourbohloul05}, 
and disruption of terrorist networks~\cite{Memon06}. 
The underlying networks may represent transportation networks, as well as social or activity networks. 
Recent interest in the problem has been in part due to the
threat of smuggling of nuclear materials and devices~\cite{pan-2003-models}.
Interdiction corresponds to the installation of special radiation-sensitive detectors across transportation links.

The problem is often posed in terms of two agents called ``interdictor'' and ``evader'' where
the evader attempts to minimize some objective function in the network, \emph{e.g.} 
distance, cost, or risk when traveling from network location $s$
to location $t$, while the interdictor attempts to limit success
by removing network nodes or edges. The interdictor has
limited resources and can thus only remove a finite set of nodes or
edges.  In the simplest formulation, 
the interdictor seeks to identify
a set of edges (or nodes) on the network whose removal maximizes the
cost of the least-cost path from a source to a destination node, while
the evader seeks to find and traverse the best unimpeded path. 
This interdiction problem is known as the ``most vital edges'' 
(or ``most vital nodes'') problem~\cite{corley-1982-most} 
and it has been shown to be NP-hard~\cite{bar-noy-1995-complexity} 
and NP-hard to approximate to better than a factor of $2$~\cite{boros06-inapproximability}. 
Methods for solving network interdiction problems have included exact algorithms
for solving integer programs, such as branch-and-bound, as well as
decomposition methods to rebuild the network by iteratively adding
relevant paths to reduce the size of both the underlying network and the number of
binary decision variables. 
A more recent approach, based on structure-dependent cutting planes,
exploits the relationship between the ordered set of evasion paths
and binary interdiction variables~\cite{pan-2007-minmax}.

A common assumption in many studies is that there is perfect
knowledge of hard-to-compute network parameters, such as the cost
to the evader to traverse an edge in terms of resource consumption
or probability of detection. However, it is clear that the evader,
and, to a lesser extent, the interdictor, have unreliable and incomplete 
information about the network. 
These uncertainties place the interdiction problem within stochastic
optimization, where one seeks to find those edges that are vital
\emph{on average}.  
Indeed, under uncertainty the evader must be described in
probabilistic terms.  By constructing such probabilistic evader models
one can expect to develop more robust interdiction solutions. 
The problem of stochastic interdiction has been the focus of a number of recent
studies~\cite{Morton-2007-models,atkinson07,Bayrak08,Janjarassuk08,reich08,Gut-09,Dim-09-MDP}.

Failure to account for evader uncertainty can lead to suboptimal decisions, namely, solutions that do not maximize (and even decrease) the evader's
expected cost to reach the target.  
Consider for instance the network in Fig.~\ref{fig:bad-classical}.
There are four paths from the source to the target: one each through nodes $1,2,3$ and the
one direct path $(0,5)$ with costs $9, 8, 8$ and $8.01$, respectively. 
If only one edge can be removed, the solution in the least-path-cost 
formulation is to remove edge $(4,5)$ which
increases the path cost from $8.0$ to $8.01$.  
However if the evader is unable to determine which path has the 
least cost and takes any path with equal (or nearly equal) probability, 
then this solution is not optimal.
% whose cost is $\leq1.25$ times the cost of the least-cost path.  
Interdiction at $(4,5)$ actually 
\emph{decreases} the expected cost from $\approx 8.25$ to $8.01$,
because it removes the costly path through node $1$.  The
optimal choice is interdiction of any one of the edges $(0,2)$, $(2,4)$,
$(0,3)$, or $(3,4)$, which increases the expected cost
from $\approx 8.25$ to $\approx 8.33$.

\begin{figure}[htbp]
\begin{center}
\includegraphics[width=1.0\columnwidth]{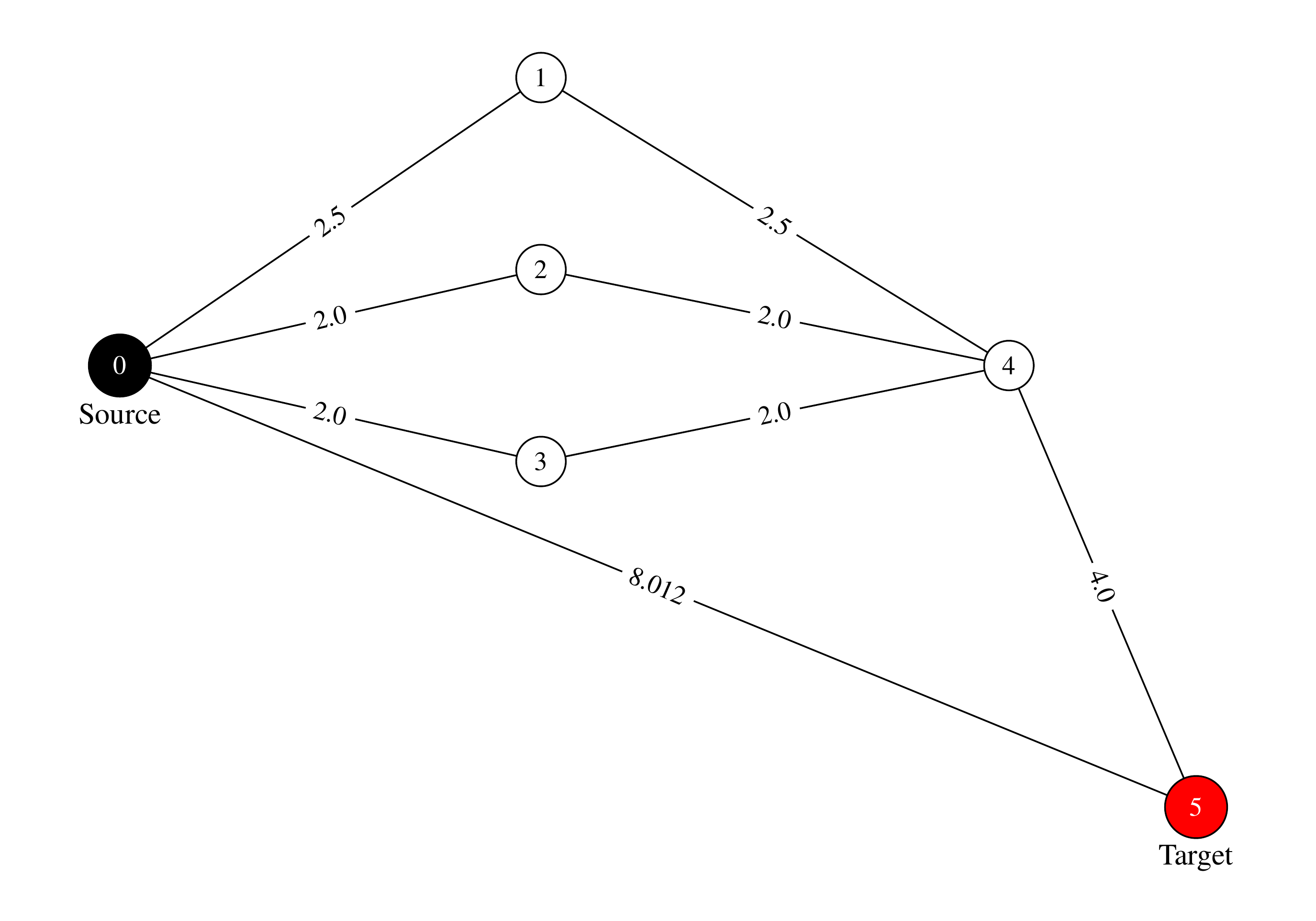}
\end{center}
\caption{Example network where the shortest path interdiction
 formulation produces a suboptimal solution when interdicting a single edge.  
 Interdicting that edge $(4,5)$ decreases the expected path cost.
 Interdicting any one of $(0,2)$, $(2,4)$,$(0,3)$, or $(3,4)$
 increases the expected path cost.
}
\label{fig:bad-classical}
\end{figure}

In this paper we describe a Markovian network interdiction framework which can
capture a wide range of network evader behavior
(Sec.~\ref{sec:general-Model}). 
We then demonstrate the general framework with a simple model based on
evader decision-making mechanisms (Sec.~\ref{sec:Evader-Decision-Model}).
Finally we develop efficient heuristic algorithms for the interdiction
problem based on the structure of the graph 
and then present performance results comparing various heuristic methods
(Sec.~\ref{sec:Solving-the-Markovian}).

\section{The interdiction model}
\label{sec:general-Model}

Our interdiction formulation is a stochastic generalization of the 
max-min shortest path interdiction problem (termed the ``least-cost path'' interdiction problem, to be exact)~\cite{ghare-1971-optimal,mcmasters-1970-optimal, isr-02}.
In the least-cost path formulation an evader attempts to traverse a network
on a path from an origin $s$ to a destination $t$.  
Let $p$ be some path between $s$ and $t$ in a 
graph $G(N,A)$ with the set of nodes $N$ and the set of
weighted edges $A$.  Let $c(p)$ be the path cost computed by 
summing the cost $C_{ij}$ over the edges $(i,j)$ of $p$, and 
 any self-looped edge has zero cost, $C_{ii} = 0$ . 
The edge costs are assumed to be given in the problem
and may depend on direction (in the case that $G(N,A)$ is a directed graph).
Here ``edge cost'' is used interchangeably with ``edge weight''.

The network interdiction strategy is represented by an interdiction set $\mathcal{R}$ 
which is a subset of the edge set $A$ of $b$ (budget).  
The decision variable $r_{ij}$ is set to $1$ if edge $(i,j)\in\mathcal{R}$, i.e. $(i,j)$ is interdicted, and $r_{ij}=0$ otherwise.  
Interdiction increases the cost of traversing $(i,j)$ by a constant $D_{ij}\geq0$.
When the value of $D_{ij}$ is very large all
paths avoid the interdicted edge $(i,j)$ (assuming that there is an
alternative path) which effectively removes the edge $(i,j)$ from the
graph.  One may write $C'_{ij} = C_{ij} + r_{ij} D_{ij}$ but it is more convenient
to use $C_{ij}$ at all times to denote cost that includes possible interdiction.
This makes the matrix ${\bf C}$ a function of $r$.

In the shortest path model, the evader only travels on paths of lowest cost, 
and is fully aware of increases in edge costs caused by interdiction decisions.
This gives the optimization problem
\begin{equation}
\max_{r\in\mathcal{R}}\min_{p\in PT}c(p)\label{eq:max-min-sp}\,,
\end{equation}
where $c(p)$ is implicitly a function of $r$, and $PT$ is the set of paths from $s$ to $t$.
The above formulation is for interdiction of edges but of course,
a similar problem could be considered for node interdiction
(by introducing for all $i\in N$ node costs $D_i$ and decision variables on nodes $r_i$.)

A stochastic version of the interdiction problem can be constructed
by supposing that an evader may take any path from $s$ to $t$, 
according to some probability distribution, 
rather than always choosing a least-cost path. 
Randomness in the evader path decision is due to the lack of knowledge of how the evader travels through the network.  It is fundamentally caused by his uncertainty about interdiction decisions $r$ or
network costs, mistaken cost computations, or possibly even by intent to
increase unpredictability. 
Suppose the evader selects path $p$ with probability $P(p)$. His expected cost of traveling from $s$ to $t$ is then
\begin{equation}
E[c]=\sum_{p\in PT}P(p)c(p)\,.
\label{eq:expected-cost}
\end{equation}
The interdiction problem becomes
\begin{equation}
\max_{r\in\mathcal{R}}\sum_{p\in PT}P(p|r)c_{r}(p)\,,
\label{eq:max-E-cost}
\end{equation}
where $P(p|r)$ is now the probability of traversing a path given
the interdiction set $r$.
The conditional probability $P(p|r)$ implicitly contains the evader's strategy.
The shortest-path optimization problem (\ref{eq:max-min-sp}) is clearly
just a special instance of (\ref{eq:max-E-cost}) when the expectation is
conditioned on traversal of only least-cost paths.

To compute the expected cost $E[c]$, we rewrite it in terms of the edge costs and the number of visits to each edge.
If $F_{ij}$ is the expected number of visits of edge $(i,j)$ by an evader, then
\begin{lemma}
\begin{equation}
E[c] =\sum_{p\in PT}P(p)c(p)= \sum_{(i,j)\in A}{C_{ij}F_{ij}}.
\label{eq:expected_cost_gen}
\end{equation}
\end{lemma} 
\begin{proof}
{By definition} $F_{ij} = \sum_{p\in PT:(i,j)\in p} P(p)$, and $F_{ij}$ can in general be larger than $1$ because paths may revisit $(i,j)$. The equivalency follows as
\begin{eqnarray*}
E[c] & = &\sum_{p\in PT}P(p)c(p)\,,\\
 & = & \sum_{p\in PT}P(p) \sum_{(i,j)\in p}{C_{ij}}\,,\\
 & = &  \sum_{(i,j)\in A}{C_{ij}\sum_{p\in PT:(i,j)\in p} P(p)}\,,\\
 & = & \sum_{(i,j)\in A}C_{ij}F_{ij}.
\end{eqnarray*}
\end{proof}
The expected cost $E[c]$ is now expressed through the expected number of visits to all edges (the $F_{ij}$ values).
The latter quantity may be hard to compute in general because every evader path could in principle visit edge $(i,j)$,
while the number of possible paths can be very large and even unbounded.
Fortunately, one particular class of stochastic models - Markov chains - gives a closed-form expression for $F_{ij}$.

\section{Markovian evaders}
We model the stochastic evader as a Markov chain that has its states at the nodes of the network.
In the most general case, the chain is completely described by (1) a distribution
of starting nodes, ${\bf a}$, and (2) a Markovian transition probability matrix, ${\bf M}$. 
In the next section, we will provide derivations of ${\bf M}$ for some realistic applications
by examining the decision-making mechanisms of a rational evader frustrated by uncertain information. 
Such an evader makes transitions that tend to bring him closer to his target.

Consider for now the most general case. 
The motion of the evader is just a Markov chain with an absorbing state at the target node $t$.
An element $M_{ij}$ of his transition probability matrix is the probability of motion from node $i$ to node $j$ along edge $(i,j)$. 
The matrix ${\bf M}$ must satisfy two conditions (1) Absorption at $t$: $M_{tt} = 1$ and $M_{ti}=0$ for all $i\neq t$, 
and (2) Access to $t$: from any starting state $i\neq t$ there is a positive probability of reaching state $t$ in
a finite number of transitions.
Because of condition (1) the transition matrix of an absorbing Markov chain can be arranged into the following canonical form
\begin{equation*}
{\bf M}=
\begin{pmatrix}
\bf \hat{M} &\bf R\\
\bf 0 & 1
\end{pmatrix}\,.
\end{equation*}
Here the matrix $\bf \hat{M}$ ($n-1$ by $n-1$) contains the transition probabilities among transient states.
The matrix $\bf R$ ($n-1$ by $1$) specifies the probabilities of transition from the transient states to the absorbing state. 

Similarly, the edge cost matrix for an absorbing Markovian evader takes a specific form
\begin{equation*}
{\bf C}=
\begin{pmatrix}
\bf \hat{C} &\bf S\\
\bf Z &0
%\bf \infty & 0
\end{pmatrix}\, .
\end{equation*}
Here the matrix $\bf \hat{C}$ ($n-1$ by $n-1$) contains the costs for transition among transient states.
The matrix $\bf S$ ($n-1$ by $1$) specifies the costs for moving to the absorbing state,
while $\bf Z$ ($1$ by $n-1$) are cost for edges out of the absorbing states - those edges are never traversed.
The element $C_{tt}=0$ implies that there is no cost to remain at the target node $t$. 

Based on the matrix $\bf \hat{M}$ one constructs the {\em Fundamental Matrix} $\bf N$ of the chain:
$${\bf N} = ({\bf I} -{\bf \hat{M}})^{-1}$$.  
\begin{theorem}
Element $N_{ij}$ of the fundamental matrix gives the expected number of visits to state $j$ if starting at state $i$ 
(Theorem 11.4 in \cite{Grinstead97}.)
\end{theorem}
In general the starting state of the evader is given by a distribution $\bf a$ over the nodes.
For convenience, the absorbing node $t$ is excluded from $\bf a$, which is $n-1$-dimensional. 
The expected number of visits to $(i,j)$ before absorption at $t$ is
\begin{corollary}
\begin{equation}
%%F_{ij} = [{\bf a (I-\hat{M})^{-1}}]_i M_{ij}.
F_{ij} = [{\bf a N}]_i M_{ij}.
\label{eq:edge_visit}
\end{equation}
\end{corollary}

The expected cost $E[c]$ for a Markovian evader can be found by substituting (\ref{eq:edge_visit}) into (\ref{eq:expected_cost_gen})~\cite{Saerens09},
\begin{theorem}
\begin{equation}\label{eq:expected_cost_mat}
%%E[c] = {\bf a}({\bf I}-\bf{\hat{M}})^{-1}( {\bf M}\odot {\bf C}) {\bf e},
E[c] = {\bf a}{\bf N} diag\left[{\bf \hat{M} \hat{C}^T+RS^T}\right],
\end{equation}
where $diag\left[{\bf \hat{M} \hat{C}^T+RS^T}\right]$ denotes the column vector of the diagonal elements of matrix ${\bf \hat{M} \hat{C}^T+RS^T}$.
\end{theorem}
In a special case where the edge cost is always $1$, i.e. $C_{ij} = 1, \ \forall (i,j)\in A$, 
$E[c]$ in (\ref{eq:expected_cost_mat}) 
reduces to the well-known expression for expected time-to-absorption: ${\bf a}{\bf N}{\bf e}$.

The objective in the Markovian network interdiction problem is to maximize $E[c]$. 
In the interdiction model, edge cost depends on the interdiction variable $r$. 
In turn, the transition matrix and the fundamental matrix depend on $r$ too. 
Therefore, this results in the nonlinear optimization problem
\begin{equation}
\max_{r\in\mathcal{R}}{{\bf a}{\bf N} diag\left[{\bf \hat{M}\hat{C}^T+RS^T}\right]}\,.
\label{eq:cost-objective}
\end{equation}
This optimization problem could be termed the \emph{Single Markovian Evader Network Interdiction} problem.
The distribution of starting nodes is assumed to be given and
independent of the interdiction strategy $r$, while the ${\bf M}$
matrix is assumed to be determined as soon as the graph and $r$ are
known. 
In numerical computations the most computationally demanding part resides
in finding ${\bf aN}={\bf a}({\bf I} -{\bf \hat{M}})^{-1}$, which require Gaussian elimination in general.

The problem in (\ref{eq:cost-objective}) can be generalized for the case of multiple evaders where each evader represents a threat scenario or an adversarial group.
Each evader $k$ then has certain probability $w^{(k)}$ of occurring $(\sum_k{w^{(k)}}=1)$,
as well as a distinctive source distribution ${\bf a}^{(k)}$, target node $t^{(k)}$ and 
transition matrix ${\bf M}^{(k)}$. The generalized objective is a weighted
sum of Eq.~(\ref{eq:expected_cost_mat}) over all evaders.

\section{Evader models}
\label{sec:Evader-Decision-Model}
As was noted in the introduction the evader may often be unable to
determine correctly the least-cost path to the target because of
incomplete and inaccurate information about the network topology,
interdiction decisions, or costs along alternative paths.  We now
develop a concrete Markovian model that incorporates uncertainty in
the path of the evader.  These types of models have analogues in other
contexts.  For example, a similar model was developed for routing in
ad-hoc wireless networks.  In that application the objective is to
transmit messages through the network with short delivery leg and balanced load~\cite{Barrett-2005-WirelesRouting}.

\subsection{The least-cost-guided evader}
\label{subsec:least-cost-guided-evader}
We suppose that at each node $i$ the evader will consider several paths from $i$ to $t$
and select the one that \emph{appears} to have the lowest cost.  Putting this in the content of a Markovian model,  we define
 $p_{i}$ be the least cost path from $i$ to $t$, with cost denoted by $c(p_{i})$. 
Suppose the evader has a destination $t$ and node $j$ is any node in the neighborhood of $i$ ($j \in G_i$).
The transition probability from $i$ to $j$ is
\begin{equation}
    M_{ij} = \frac{e^{-\lambda \left(c(p_{i})-C_{ij}-c(p_{j})\right)}}{\sum_{j\in G_i}{e^{-\lambda \left(c(p_{i})-C_{ij}-c(p_{j})\right)}}}\label{eq:mij-from-cost}\,,
\end{equation}
where $\lambda \geq 0$ is a parameter (see Fig.~\ref{fig:m_illustrated}.)
\begin{figure}[htbp]
\includegraphics[width=\columnwidth]{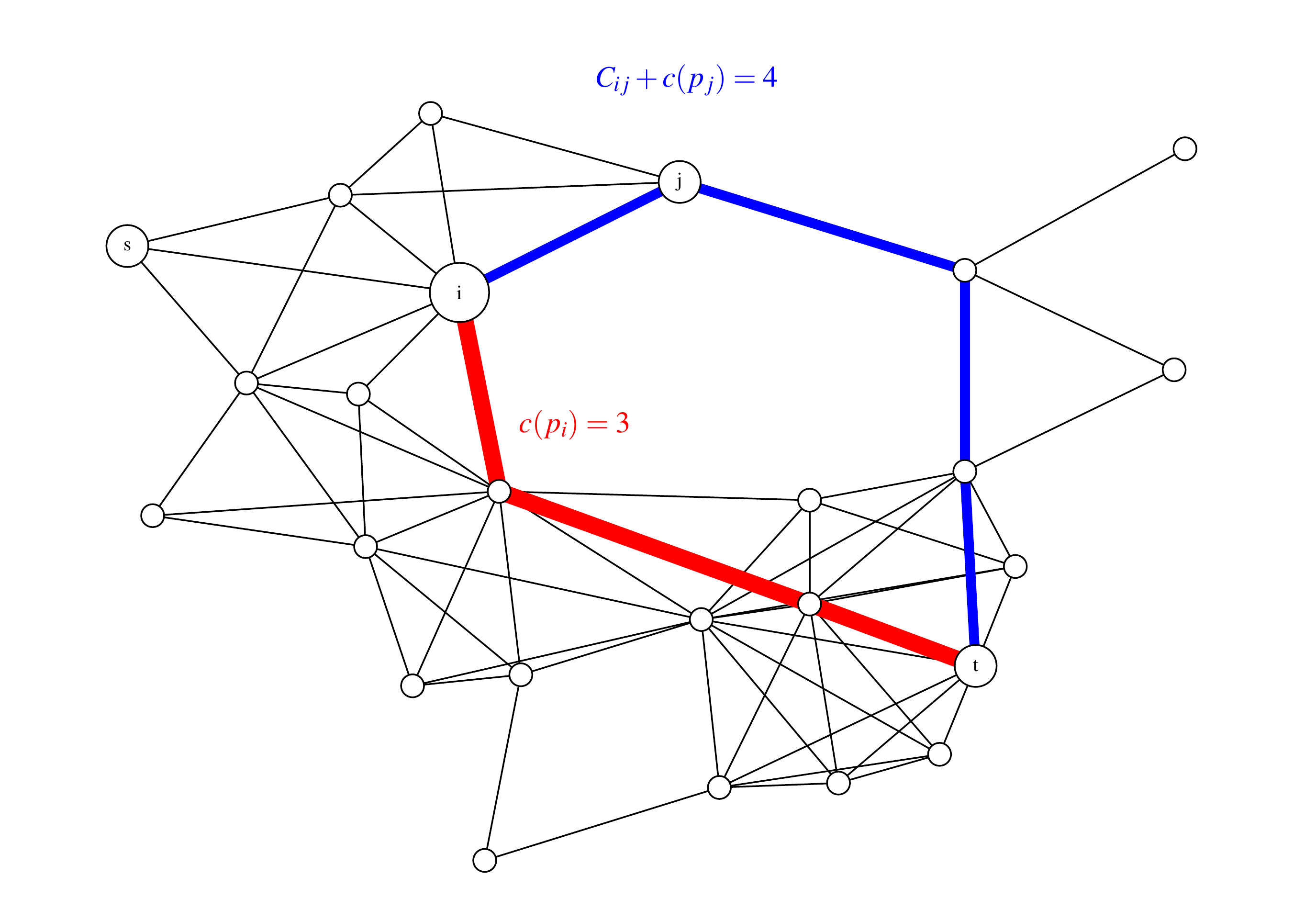}
\caption{
  Computation of the transition probabilities $M_{ij}$.  
  The least-cost path from node $i$ to the target $t$ 
  is the path $p_{i}$ (thick red) with cost $c(p_{i})=3$.  
  Through node $j$ the shortest path to $t$ is (thin blue) path 
  $p_{i}$ with cost $C_{ij}+c(p_{j})=4$.
}
\label{fig:m_illustrated}
\end{figure}

The adherence to the least-cost path is determined by the parameter $\lambda$.
When $\lambda\to\infty$ the evader moves deterministically along the least-cost path (or paths) and when
$\lambda\to0$ the motion is perfectly random. 
The least-cost path has the highest probability, but the difference with other paths vanishes as $\lambda\to0$. 
Hence, the model can be called the ``least-cost-guided evader''.

Notice that although ${M}_{ij}$ values in Eq.~(\ref{eq:mij-from-cost}) depend
on the cost of least-cost path, 
when $\lambda < \infty$ this dependence is a smooth function of path costs.
Thus the new formulation provides a more desirable description
of evader motion because it avoids the sensitivity to path costs seen in the shortest-path evader model.
The process of computing the probabilities involves running Dijkstra's algorithm to find the distance
to the target node from each node $i$, which gives $c(p_{i})$.

\subsection{The least-risk-guided evader}
In some applications the evader may base decisions on the risk of
crossing an edge rather than the cost.  In those cases the each edge
in the network is assigned a value $Y_{ij}$ for the probability of
successful evasion, instead of a cost $C_{ij}$ .  The evader attempts
to find the path to the target $t$ that offers the greatest
probability of evasion which is is just the product of those $Y_{ij}$
values along the path.

Let $q_{ij}$ be the probability of successful evasion on a path consisting of the edge $(i,j)$ and then of the least-risk path from $j$ to the target.
One choice is to assume that an evader would traverse edge $(i,j)$ with probability
\textit{proportional} to $q_{ij}$, or more generally, proportional to a positive
power of $q_{ij}$ 
\begin{equation}
    {M}_{ij}\propto\left(\frac{q_{ij}}{q_{i*}}\right)^{\lambda}\label{eq:transition-mat-q}\,,
\end{equation}
where $\lambda>0$ is a parameter, $q_{i*}=\max_j{q_{ij}}$ is the probability of
evasion if the least-risk path from $i$ to the target is followed
(the constant of proportionality is found from $\sum_{j}{M}_{ij}=1$.)

\subsection{The non-retreating evader}
A simple variant the least-cost-guided model is the non-retreating evader.
In this model it is assumed that an evader always moves to nodes that
are closer to the target node $t$ than the current node.
To represent this model assume that there is zero probability
of motion through $(i,j)$ if node $i$ is at least as close to the target as
node $j$, namely, $c(p_{i})\leq c(p_{j})$, where $c(p_{i})$ and $c(p_{j})$ 
are the smallest costs of paths to the target from nodes $i$
and $j$, respectively, computed by summing the edge weights. 

An interesting effect of this assumption is that the evader evader 
would never cross a node or an edge twice.
Consequently the set of nodes becomes a partially ordered set and 
as a result, there exists a relabeling $\sigma$ %(i.e.\ a basis)
of the nodes such that if $c(p_{i})>c(p_{j})$ then $\sigma(i)>\sigma(j)$.
A simple (non-unique) procedure is to label the target node $t$ as $0$ 
($\sigma(t)=0$) and then rank the nodes in the order of their distance 
(cost) along least-cost path to $t$, breaking ties
arbitrarily. Computationally, this is the
same as the order the nodes are reached by a shortest path
(Dijkstra's) algorithm starting at $t$.
 The transition probability becomes 
\begin{equation*}
\hat{M}_{ij}=
\begin{cases}
M_{ij}\,, & c(p_{i})>c(p_{j})\,,\\
0\,, & c(p_{i})\leq c(p_{j})\,.
\end{cases}
\end{equation*}

In this case all paths must reach
the target after at most $|N|-1$ steps, where $|N|$ is the number nodes in $G$,
and hence ${\bf \hat{M}}$ becomes nilpotent of power $|N|-1$. Moreover, by labeling
the nodes up in order of increasing cost, ${\bf \hat{M}}$ can be written
as a lower-triangular matrix with zero diagonal. For example, if
the evader traverses a $2\times 3$ grid with the target at one corner
then one possible $\sigma$ gives the matrix
\begin{equation*}
{\bf \hat{M}}=
\begin{pmatrix}
0 \\
1 & 0\\
1 & 0 & 0\\
0 & 1   & 0 & 0\\
0 & 0.5  & 0.5  & 0  & 0\\
0 &  0  &  0  & 0.5 & 0.5 & 0
\end{pmatrix}\,.
\end{equation*}

The special matrix structure facilitates an order-of magnitude speedup in the computation
of Eq.~\ref{eq:expected_cost_mat}. For a general ${\bf M}$, computing 
${\bf a}{({\bf I-M})}^{-1}$
involves Gaussian elimination at a cost of $2|N|^{3}/3$ operations.
For a nilpotent lower-triangular ${\bf \hat{M}}$ 
the cost falls to $O(|N|^{2})$ since we can use backward-forward
substitutions instead of Gaussian elimination. The cost of computing
the objective function Eq.~(\ref{eq:expected_cost_mat}) is also expected to drop to $O(|N|^2)$
despite the need to reorder the matrix ${\bf C}$ when the nodes are
relabeled.

\section{Solving the Markovian interdiction problem}
\label{sec:Solving-the-Markovian}

The challenge of network interdiction consists of developing both
realistic models and tractable algorithms.
The Markovian evader model adds realism 
but does not reduce the computational complexity of finding good interdiction solutions.
Indeed it is clear that the Markovian model is computationally hard because 
in the limit of $\lambda\to\infty$, the model becomes the least-cost
interdiction problem which is NP-Hard \cite{Ball89,bar-noy-1995-complexity}
and also hard to approximate~\cite{boros06-inapproximability}. 
Therefore, this section discusses solution heuristics based on  
network structure. 

A common approach to solving many combinatorial optimization problems
is based on local, or neighborhood, search algorithms such as simulated annealing~\cite{Kelly96}.
But those general-purpose local search algorithms 
do not scale well to larger problems or find poor solutions.
The solution space may be exponential in the budget so any iterative improvement process 
of local search can only explore a very small fraction of solutions in a polynomial number
of steps.  
It follows that high-quality solutions can only come from more
specialized solvers that exploit the structure of the interdiction problem.  
We explore algorithms based on ranking functions that rank edges according to global information about graph structure.

\subsection{Betweenness centrality heuristic}
The most successful ranking function we found is derived from the
shortest-path betweenness centrality.
The shortest-path betweenness centrality of an edge is the fraction of
least-cost paths between all pairs of nodes in a network that cross 
the edge~\cite{freeman-1977-set}.  This metric identifies those
edges that are critical to connectivity within a network, such as
bridge edges that joins two graph components, because they participate
in a large number of least-cost paths linking nodes on a network,

We constructed an heuristic based on shortest-path betweenness centrality by considering
only paths between the sources {\bf a} and the target $t$ of the evader.
Recall that $a_s$ is the probability that the evader would start at node $s$.
Let $\sigma_{st,\mathcal{R}}$ be the number of least-cost paths between nodes $s$ and the target node $t$ in the graph with interdiction set $\mathcal{R}$.
Similarly, let $\sigma_{st,\mathcal{R}}(e)$ be the number of those paths that pass through edge $e$.
Therefore, we define the source-weighted centrality of edge $e$ with respect to $t$ as the sum
\begin{equation}
    H_{\mathcal{R}}(e)= \sum_{s : t \ne s \in V} a_s \frac{\sigma_{st,\mathcal{R}}(e)}{\sigma_{st,\mathcal{R}}}.
\end{equation}
Notice that this quantity needs to be re-computed during execution of
an interdiction problem:
as the interdiction set $\mathcal{R}$ is increased, the costs of the edges change
and so are the least-cost paths.
An algorithm for calculating a metric of this kind for all $e\in A$ 
in $O(|A|+|N|\log|N|)$ time is found in Ref.~\cite{Brandes-2001-faster}.
In the case of multiple evaders, the heuristic is computed for each
evader and weighted based on $w^{(k)}$.

\subsection{Algorithms}
We use the betweenness heuristic $H_{\mathcal{R}}(e)$ to rank the edges $e$ in the network
given the interdiction set $\mathcal{R}$.  This heuristic leads to a
simple algorithm, termed Betweenness (Alg.~\ref{al:DGH}), that performs
a sequential selection of edges. 
\begin{algorithm}[ht]\caption{Betweenness algorithm using global heuristic $H$ for budget $B$}
    \label{al:DGH}
\begin{algorithmic}

\STATE $\mathcal{R}\leftarrow\varnothing$

\WHILE{$B>0$}

\STATE $\mathcal{R}\leftarrow \mathcal{R}\cup\left\{ \argmax_{e\in A\smallsetminus \mathcal{R}} H_{\mathcal{R}}(e)\right\} $,
resolving ties arbitrarily.

\STATE $B\leftarrow B-1$

\ENDWHILE
\STATE \textbf{Output}($\mathcal{R}$)
\end{algorithmic}
\end{algorithm}
The betweenness algorithm is fast since it does not evaluate the objective function but
only has to initially compute the ranking heuristic and then re-evaluate it after
the interdicted edge is chosen.  The heuristic is called $B$ times: once for each of the budgeted edges.

For comparison we also use a more computational expensive greedy
algorithm (Alg.~\ref{al:RG}) that constructs the interdiction set $\mathcal{R}$
incrementally.  At each of the $B$ steps, the greedy algorithm
computes $\Delta_{\mathcal{R}}(e)$, the increase in the objective function due
to addition of edge $e$ and then selects the best edge.
\begin{algorithm}[ht]
\caption{Greedy algorithm for the construction of the interdiction set $\mathcal{R}$ with budget $B$}
\label{al:RG}
\begin{algorithmic}

\STATE $\mathcal{R}\leftarrow\varnothing$

\WHILE{$B>0$}
\FORALL {$e\in A$}
    \STATE $\Delta_{\mathcal{R}}(e):=h\left(\mathcal{R}\cup\left\{ e\right\} \right)-h\left(\mathcal{R}\right)$
\ENDFOR

\STATE $\mathcal{R}\leftarrow \mathcal{R}\cup\left\{ \argmax_{e\in A}\Delta_{\mathcal{R}}(e)\right\} $, resolving ties arbitrarily.

\STATE $B\leftarrow B-1$

\ENDWHILE
\STATE \textbf{Output}($\mathcal{R}$)
\end{algorithmic}
\end{algorithm}

\subsection{Performance results}
We now demonstrate the performance of the Greedy and Betweenness
algorithms on a sample network interdiction problem and show the
effect of varying the randomness parameter $\lambda$.  We used a
network which consists
of a $10\times10$ grid of directed edges with 10 added shortcuts
between random pairs of nodes for a total of $420$ edges.
Weights were assigned to each edge by
choosing uniformly at random from the interval [0.5, 1.5].  We
selected 2 distinct targets at random (i.e. 2 evaders) each with 5 source
locations.
\begin{figure}[H]
\includegraphics[width=\columnwidth]{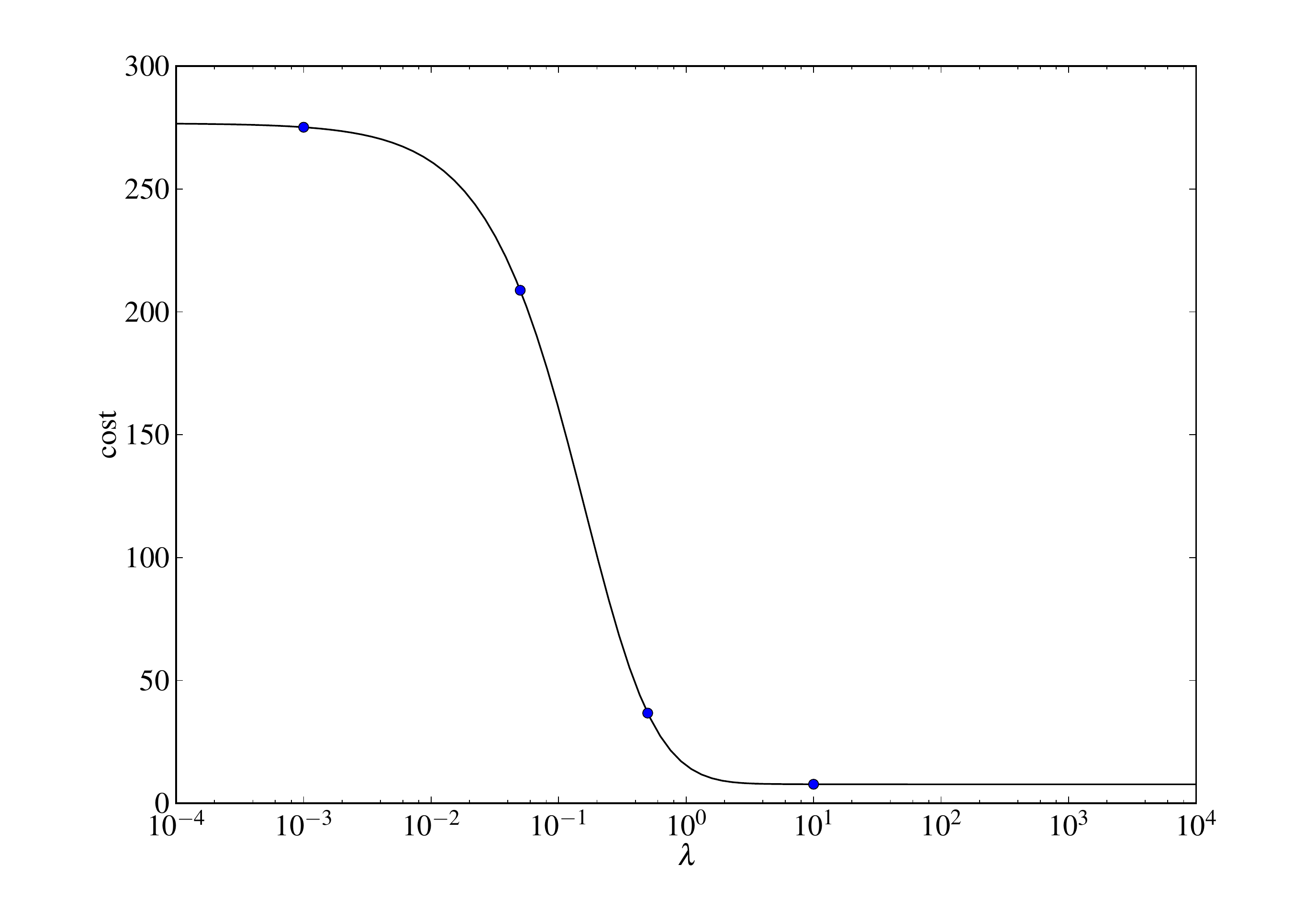}
\caption{
  The expected cost of reaching the target from the source a function of the
  parameter $\lambda$ for an example network.   For large values of $\lambda$ the
  model chooses only the shortest path and the expected cost is lowest. As $\lambda$ decreases
  the cost increases as the paths become more random.  For $\lambda=0$ the paths are completely
  random and the cost is at the maximum.
  The expected cost is calculated by Eq.~(\ref{eq:expected_cost_mat}) with the evader model
  $M$ given by Eq.~(\ref{eq:mij-from-cost}).
  The network is a $10\times10$ directed grid with $10$ randomly added shortcut
  edges and the target and source are chosen randomly.  Each of the
  edges have weights chosen uniformly from [0.5,1.5].  
  The marked points will be used in performance evaluations, presented in Fig.~\ref{fig:heuristics}.
}
\label{fig:baseCost_vs_lambda}
\end{figure}

The motion of the evader followed the least-cost-guided model.  
In this model, the effect of the parameter $\lambda$ on the expected cost for the
evader (before interdiction) is not linear, as shown in Fig.~\ref{fig:baseCost_vs_lambda}.
At low values of $\lambda$ 
the motion is random and the cost is the highest.  As $\lambda$ is
increased the evader follows paths that are closer to the optimal
path and the cost decreases continuously toward the minimum
achievable at large $\lambda$.  The transition between the cost of
random motion and the optimal cost occurs rapidly over a small
range of $\lambda$ where the most diverse behavior is found.
This transition in behavior was observed in other random and structured graphs and
real-world networks that we examined and is a feature of the nonlinear
dependence of the path probabilities from Eq.~(\ref{eq:mij-from-cost}).

Fig.~\ref{fig:heuristics} shows characteristic performance results for both the Greedy and Betweenness algorithms for various $\lambda$.  
The performance is measured in terms of the expected cost given by
Eq.~(\ref{eq:expected_cost_mat}).  Interdiction of an edge causes the
weight of the edge to increase by a fixed value $D_{ij}$.  We set the added
increase to be half the diameter of the network which in this case is
$D_{ij}=4.5$.
\begin{figure*}[!t]
  \includegraphics[width=2\columnwidth]{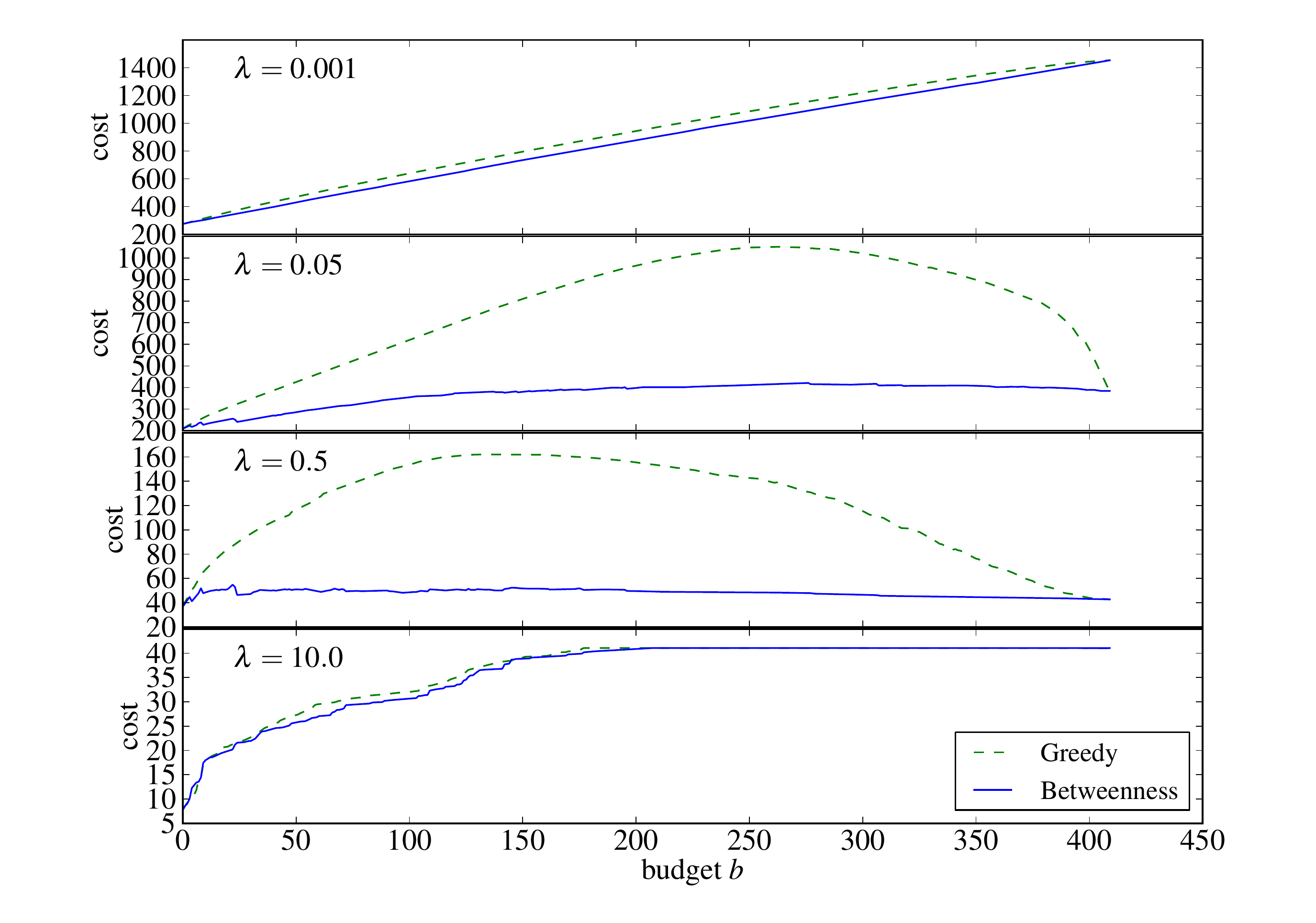}
  \caption{\label{fig:heuristics}
    Comparison of the Greedy (\ref{al:RG}) and Betweenness
(\ref{al:DGH}) algorithms for given budgets on the sample grid network
described in Fig.~$\ref{fig:baseCost_vs_lambda}$.  Four different
values of $\lambda$ are shown corresponding to different levels of
randomness in the evader path selection. When the randomness of the
evader is low (high $\lambda$) the Betweenness algorithm performs very
well compared to the higher computational cost Greedy algorithm.  As
the randomness increases the algorithms' performance diverges after very
small budgets - demonstrating that the Betweenness heuristic is no
longer effective.  At low values of $\lambda$ the evader motion is
random and no algorithm will be effective.  The convergence of the
algorithms at large budgets occurs because we do not allow an
edge to be interdicted more than once and at that budget every edge in
the graph is interdicted and the costs are the same.
  }
\end{figure*}

For small budgets the Betweenness algorithm and Greedy algorithm
produce comparable results as measured by the increase in cost for all
$\lambda$ values.  The Betweenness algorithm is considerably cheaper
in computational cost. As the budget is increased the Betweenness heuristic
performs very well for larger $\lambda$.
But for smaller $\lambda$, as the evader randomness
increases, the algorithm performance difference diverges indicating
that the Betweenness heuristic is no longer effective.  At very low
values of $\lambda$ the evader motion is random and no algorithm is expected to
be effective.

A particularly interesting phenomenon is the non-monotonicity of the expected cost.
Namely, for some low $\lambda$ values the expected cost $E(c)$ sometimes actually decreases after 
the interdiction set is enlarged.
This effect was anticipated by the example in Fig.~\ref{fig:bad-classical}
and it occurs because the behavior of the randomizing evader is fundamentally
different from the behavior of the max-min evader.
If we relax the budget constraint $|\mathcal{R}| = B$ to $|\mathcal{R}| \le
B$, the objective will be nondecreasing in the Greedy algorithm.

Other realizations of $10\times10$ grid networks produce
similar results and are not shown here. 
In addition to this example we have explored the performance of the
algorithms on other networks including real-world of
transportation networks, such as the Washington DC transportation
transit time network and the Rome city road
network~\cite{gutfraind-transport-2010}. The computation cost of the
Greedy algorithm becomes prohibitive in these and other 
urban, national and international transportation systems.
Those networks have $10^3 - 10^7$ edges, depending on the spatial resolution.
The Greedy algorithm running time scales as $O(|A||N|^3)$ for the least-cost-guided evader model, while the Betweenness algorithm remains feasible even on very large instances because its
running time scales as $O(|A|+|N|\log|N|)$.

\section{Conclusions and outlook}
Practical instances of network interdiction must invariably address
the uncertainty in the network structure and evader behavior.  Such
behavior can be modeled using the proposed Markov chain approach, which achieves increased
realism while remaining analytically penetrable.  To summarize, the main
contribution of this work are:
\begin{itemize}
    \item a demonstration of the fundamental advantages of stochastic models over least-cost models,
    \item a stochastic model of the evader motion based on a Markovian guided random walk, and
    \item a scalable interdiction algorithm based on a specialized betweenness centrality function.
\end{itemize}

Future research must address both computational and modeling
challenges in stochastic network interdiction.  Current algorithms are
effective in the case where the evader motion is partially
predictable.  It is not known whether more specialized heuristics can
be more successful in the case of highly-stochastic adversaries.  In
the current model the randomness comes only from information
constraints. In some problems computational constraints on the evader also play a role in
determining his motion.  Models that account for both kinds of constrains
promise further gains in realism and would expand the range of
applications where network interdiction could be used.

\paragraph*{Acknowledgments}
AG would like to thank David Shmoys and Vadas Gintautas for fruitful
discussions.  Part of this work was funded by by DTRA Basic Research
under contract IACRO \#09-4693I.

\pagebreak[3]
\bibliographystyle{plain}
\bibliography{interdict}

\end{document}